\documentclass{amsart}

\usepackage[all]{xy}
\usepackage{amsmath}
\usepackage{hyperref}
\usepackage{amsfonts,graphics,amsthm,amsfonts,amscd,latexsym}
\usepackage{epsfig}
\usepackage{flafter}
\usepackage{mathtools}
\usepackage{comment}

\hypersetup{
    colorlinks=true,    
    linkcolor=blue,          
    citecolor=blue,      
    filecolor=blue,      
    urlcolor=blue           
}
\usepackage{tikz}
\usetikzlibrary{graphs,positioning,arrows,shapes.misc,decorations.pathmorphing}

\tikzset{
    >=stealth,
    every picture/.style={thick},
    graphs/every graph/.style={empty nodes},
}

\tikzstyle{vertex}=[
    draw,
    circle,
    fill=black,
    inner sep=1pt,
    minimum width=5pt,
]
\usepackage[position=top]{subfig}
\usepackage[alphabetic,backrefs]{amsrefs}
\usepackage{amssymb}
\usepackage{color}

\calclayout

\usetikzlibrary{decorations.pathmorphing}
\tikzstyle{printersafe}=[decoration={snake,amplitude=0pt}]

\newcommand{\Spec}{\operatorname{Spec}}

\newcommand{\pp}{\mathbb{P}}

\renewcommand{\qq}{\mathbb{Q}}
\newcommand{\zz}{\mathbb{Z}}

\newcommand{\cc}{\mathbb{C}}
\newcommand{\kk}{\mathbb{K}}

\def\O#1.{\mathcal {O}_{#1}}			
\def\pr #1.{\mathbb P^{#1}}				
\def\af #1.{\mathbb A^{#1}}			
\def\ses#1.#2.#3.{0\to #1\to #2\to #3 \to 0}	
\def\xrar#1.{\xrightarrow{#1}}			
\def\K#1.{K_{#1}}						
\def\bA#1.{\mathbf{A}_{#1}}			
\def\bM#1.{\mathbf{M}_{#1}}				
\def\bL#1.{\mathbf{L}_{#1}}				
\def\bB#1.{\mathbf{B}_{#1}}				
\def\bK#1.{\mathbf{K}_{#1}}			
\def\subs#1.{_{#1}}					
\def\sups#1.{^{#1}}						

\DeclareMathOperator{\coeff}{coeff}	

\DeclareMathOperator{\Supp}{Supp}

\usepackage{tikz}
\usetikzlibrary{matrix,arrows,decorations.pathmorphing}

  \newtheorem{theorem}{Theorem}[section]
  \newtheorem{lemma}[theorem]{Lemma}
  \newtheorem{proposition}[theorem]{Proposition}
  \newtheorem{corollary}[theorem]{Corollary}

  \newtheorem{notation}[theorem]{Notation}
  \newtheorem{definition}[theorem]{Definition}
  \newtheorem{example}[theorem]{Example}

\newtheorem{remark}[theorem]{Remark}

\theoremstyle{remark}

\numberwithin{equation}{section}

\begin{document}

\title[Bounded deformations of $(\epsilon,\delta)$-log canonical singularities]{Bounded deformations of $(\epsilon,\delta)$-log canonical singularities}

\author[J.~Han]{Jingjun Han}
\address{
Department of Mathematics, Johns Hopkins University, Baltimore, MD 21218, USA}
\email{jhan@math.utah.edu}

\author[J.~Liu]{Jihao Liu}
\address{
Department of Mathematics, University of Utah, 155 S 1400 E, JWB 233,
Salt Lake City, UT 84112, USA}
\email{jliu@math.utah.edu}

\author[J.~Moraga]{Joaqu\'in Moraga}
\address{
Department of Mathematics, University of Utah, 155 S 1400 E, JWB 233,
Salt Lake City, UT 84112, USA}
\email{moraga@math.utah.edu}

\date{\today}

\maketitle

\begin{abstract}
In this paper we study $(\epsilon,\delta)$-lc singularites, i.e.
$\epsilon$-lc singularities admitting a $\delta$-plt blow-up.
We prove that $n$-dimensional $(\epsilon,\delta)$-lc singularities are bounded up to a deformation, and $2$-dimensional $(\epsilon,\delta)$-lc singularities form a bounded family.
Furthermore, we give an example which shows that $(\epsilon,\delta)$-lc singularities
are not bounded in higher dimensions, even in the analytic sense. 
\end{abstract}

\setcounter{tocdepth}{1}
\tableofcontents

\section{Introduction}
Throughout this paper, we work over an algebraically closed field $\mathbb{K}$ of characteristic zero.

Log canonical singularities (lc singularities for short) are the main class of singularities of the minimal model program.
In many problems of birational geometry, 
we are interested in lc singularities whose log discrepancies are greater than zero. These singularities are called Kawamata log terminal singularities (klt singularities for short), which can be viewed as the local version of log Fano varieties ~\cite[3.1]{Bir18}.
Indeed, given a klt singularity $x\in X$, there exists a birational modification $\pi \colon Y \rightarrow X$ which extracts a unique divisor $E$ mapping onto $x\in X$. $E$ is a log Fano variety (c.f.~\cite{Sho96,Pro00,Kud01,Xu14}), hence many properties of $E$, such as log discrepancies and complements, are reflected on the singularity $x\in X$. Indeed, the boundedness of $\epsilon$-lc Fano varieties due to Birkar~\cite{Bir16a,Bir16b} has implications on the control of invariants of singularities of $x\in X$. More precisely, for singularities $x\in X$ of fixed dimension $n$ whose log discrepancies are $\geq\epsilon$ and admitting a $\delta$-plt blow-up, the Cartier indices of any $\mathbb Q$-Cartier Weil divisor near $x\in X$ and the minimal log discrepancies of $X$ at $x\in X$ are both contained in a finite set which only depends on $n,\epsilon$ and $\delta$ (c.f.~\cite{Sho00,Mor18a,Mor18b,HLS19}). 

The singularities above will be called $n$-dimensional \emph{$(\epsilon,\delta)$-log canonical singularities} ($(\epsilon,\delta)$-lc singularities for short) in this paper.  It is then natural to ask whether $(\epsilon,\delta$)-lc singularities are contained in a bounded family, even in the analytic sense. Unfortunately, this question has a negative answer even in dimension $3$, as shown in Example~\ref{unbounded} below. Nevertheless, we are able to show that $(\epsilon,\delta)$-singularities are contained in a bounded family up to deformation, which is the main result of our paper:

\begin{theorem}\label{thm:bounded-deformations}
Let $n$ be a positive integer, $\epsilon$ and $\delta$ two positive real numbers. Suppose $\mathcal{D}_n(\epsilon,\delta)$ is the class of $n$-dimensional $\qq$-Gorenstein $(\epsilon,\delta)$-lc singularities.

Then $\mathcal{D}_n(\epsilon,\delta)$ is a bounded family up to a deformation.
\end{theorem}

The above theorem still hold if we consider pairs $(X,B)$ such that the coefficients of $B$ are $\geq c$ for some fixed positive real number $c$. See Theorem~\ref{thm:bounded-deformations-boundary} for the precise argument.

For $(\epsilon,\delta)$-surface singularities, we can show that they are contained in an analytically bounded family. More precisely, we have the following:

\begin{corollary}\label{thm:bounded-surfaces}
Suppose we work over the field of complex numbers $\mathbb C$. Let $\epsilon,\delta$ be two positive real numbers, and $\mathcal{D}_2(\epsilon,\delta)$ the set of all $(\epsilon,\delta)$-lc surface singularities. Then there exists a positive real number $\epsilon_0$ and a positive integer $N$ depending only on $\epsilon$ and $\delta$ satisfying the following. Assume 
$$\mathcal{C}_{2}(\epsilon_0,N):=\{(x\in X)\in\mathcal{D}_2(\epsilon_0,0)\mid  
\text{
$x\in X$ is a cone singularity with isotropies at most $N$}\},$$
then
\begin{enumerate}
    \item $\mathcal{C}_{2}(\epsilon_0,N)$ is finite,
    \item $\mathcal{D}_2(\epsilon,\delta)$ is analytically bounded, and
    \item any element of $\mathcal{D}_2(\epsilon,\delta)$ degenerates to an element of $\mathcal{C}_{2}(\epsilon_0,N)$.
\end{enumerate}
\end{corollary}

This corollary is proved using Theorem~\ref{thm:bounded-deformations}, classic computations for surface singularities, and the theory of deformationof quasi-homogeneous surface singularities developed by Schlessinger, Pinkahm, Wahl, etc. (c.f.~\cite{Pin74, Sch73, Wah76}).

As a consequence of Theorem \ref{thm:bounded-deformations}, we show that any lower-semicontinuous (resp. upper-semicontinuous) invariant of singularities has a lower bound (resp. upper bound) for $(\epsilon,\delta)$-lc singularities:

\begin{theorem}\label{lower-bound}
Let $n$ be a positive integer,  $\epsilon$ and $\delta$ two positive real numbers. Let $i$ be a lower-semicontinuous (resp. upper-semicontinuous) invariant of klt singularities.
Then there exists a constant $i_0$ depending only on $n,\epsilon,\delta$ and $i$, such that for any $n$-dimensional $(\epsilon,\delta)$-lc singularity $x\in X$, $i(x\in X)\geq i_0$ (resp. $i(x\in X)\leq i_0)$.
\end{theorem}

As an immediate corollary, we show that multiplicities of $(\epsilon,\delta)$-lc singularities of fixed dimension are bounded from above:

\begin{corollary}\label{upper-bound-mult}
Let $n$ be a positive integer, $\epsilon$ and $\delta$ two positive real numbers. Then there exists a positive real number $m$ depending only on $n,\epsilon$ and $\delta$, such that the multiplicity of any $n$-dimensional $(\epsilon,\delta)$-lc singularity is at most $m$.
\end{corollary}

Moreover, we can control the analytic embedding dimension of $n$-dimensional complex $(\epsilon,\delta)$-log canonical singularities.

\begin{corollary}\label{upper-bound-analytic-embedding-dimension}
Let $n$ be a positive integer, $\epsilon$ and $\delta$ two positive real numbers. Then there exists a positive real number $e$ depending only on $n,\epsilon$ and $\delta$, such that the analytic embedding dimension of any $n$-dimensional complex $(\epsilon,\delta)$-lc singularity is at most $e$.
\end{corollary}

We give a brief sketch of the proof of Theorem~\ref{thm:bounded-deformations}.
Let $x\in X$ be an $(\epsilon,\delta)$-singularity of dimension $n$, then there exists a $\delta$-plt blow-up $\pi: Y\rightarrow X$ which extracts a unique divisor $E$. Motivated by the cone singularity degeneration described in~\cite{LX16,LX17}, which is a special case of the normal cone deformation~\cite{Ful98}, we degenerate the singularity $x\in X$ to the cone over $E$ with respect to the $\qq$-polarization induced by the anit-co-normal sheaf $-E|_{E}$. This gives us a flat deformation $\mathcal{X}\rightarrow \mathbb A^1_\kk$ whose fiber over a general point is isomorphic to $X$, and whose fiber over the origin is isomorphic to a cone singularity $x_0\in X_0$ in the classic sense of~\cite{Dem88}.

We show that $E$ has bounded Cartier index on $Y$ via standard arguments of the minimal model program, which makes us possible to bound isotropies of the natural $\kk^*$-action on $X_0$ and the log discrepancies at $x_0\in X_0$. The cone singularity $x_0 \in X_0$ appearing as central fibers of the above deformation, is now bounded according to \cite[Theorem 1]{Mor18b}. In particular, we deduce that $n$-dimensional $(\epsilon,\delta)$-lc singularities are bounded up to deformation.

\subsection*{Acknowledgements}
The authors would like to thank Harold Blum, Christopher Hacon, Chi Li, Yuchen Liu, Lu Qi, V.V. Shokurov and Chenyang Xu for many useful comments.
The second and third author were partially supported by NSF research grants no: DMS-1300750, DMS-1265285 and by a
grant from the Simons Foundation; Award Number: 256202.

\section{Preliminaries}

In this section we introduce some definitions and prove some preliminary results that will be used in the proof of the main theorems. We adopt the standard notions and conventions in ~\cite{KM98,HK10,Kol13}, and will freely use them.

\subsection{Singularities}

In this subsection, we recall the classic definitions of singularities of the minimal model program and introduce $(\epsilon,\delta)$-lc singularities. 
\begin{definition}[Pairs]{\em 
A {\em couple} $(X,D)$ consists of a normal variety $X$ and a reduced divisor $D$ on $X$.  A {\em sub-pair} $(X,B)$ consists of a normal quasi-projective variety $X$ and an $\mathbb R$-divisor $B$ on $X$ such that $K_X+B$ is $\mathbb R$-Cartier. A sub-pair $(X,B)$ is called a {\em pair} if $B\geq 0$.

Let $(X,B)$ be a pair, $E$ a prime divisor on $X$, and $D$ an $\mathbb R$-divisor on $X$. We define $\coeff_ED$ to be the multiplicity of $E$ along $D$. For any divisor $D$ over $X$. Let $\pi:W\to X$
be a log resolution of $(X,B)$ such that $D$ is a divisor on $W$, and suppose that
$$K_W+B_W=\pi^{*}(K_X+B).$$
The \emph{log discrepancy} of $D$ on $W$ with respect to $(X,B)$ is defined as $1-\coeff_{D}B_W$ and is denoted by $a_D(X,B).$

Let $\epsilon$ be a nonnegative real number. We say $(X,B)$ is log canonical (resp. Kawamata log terminal, $\epsilon$-log canonical) if $a_D(X,B)\ge0$ (resp. $>0$, $\ge \epsilon$) for every prime divisor $D$ over $X$. We say $(X,B)$ is purely log terminal (resp. $\epsilon$-purely log terminal) if $a_D(X,B)>0$ (resp. $>\epsilon$) for any exceptional prime divisor $D$ over $X$. For simplicity, we shall write lc (resp. klt, plt) for log canonical (resp. Kawamata log terminal, purely log terminal) in the rest of the paper.}\end{definition}

\begin{definition}[$\delta$-plt blow-up]{\em
Let $\delta$ be a nonnegative real number, $(X,B)$ an lc pair, and $x\in X$ a closed point. A {\em $\delta$-plt blow-up} of  $(X,B)$ at $x\in X$ is a birational morphism $\pi \colon Y\rightarrow X$, such that
\begin{itemize}
\item $\pi$ has a unique exceptional divisor $E$, such that ${\rm center}_XE=\{x\}$,
\item $-E$ is ample over $X$, and 
\item $(Y,B_Y+E)$ is $\delta$-plt near $E$.
\end{itemize}

A $0$-plt blow-up is also called a {\em plt blow-up}.
}
\end{definition}

\begin{definition}{\em 
A pair $(X,B)$ is called {\em $(\epsilon,\delta)$-lc} at $x\in X$
if $(X,B)$ is $\epsilon$-lc near $x$ and
there exists a $\delta$-plt blow-up of $(X,B)$ at $x$.

}
\end{definition}

\begin{remark}{\em 
Any klt singularity is $\epsilon$-lc for some positive real number $\epsilon$. Moreover, since any klt singularity admits a plt blow-up (c.f.~\cite[3.1]{Sho96},~\cite[2.9]{Pro00},~\cite[1.5]{Kud01},~\cite[Lemma 1]{Xu14}), any klt singularity admits a $\delta$-plt blow-up for some positive real number $\delta$. Thus, any klt singularity is $(\epsilon,\delta)$-lc for some positive real numbers $\epsilon$ and $\delta$.}
\end{remark}

\subsection{Deformation, families and boundedness}

In this section, we recall the definitions of deformations, families and boundedness. We also prove a proposition on log boundedness for certain pairs which will be used later.

\begin{definition}{\em 

Let $(X,B)$ be a pair. A {\em deformation} of $X$ is a flat morphism $\mathcal{X}\rightarrow T$ such that $\mathcal{X}_0 \simeq X$. A {\em deformation} of $(X,B)$ is a flat morphism $\mathcal{X}\rightarrow T$ and a divisor $\mathcal{B}\subset \mathcal{X}$, such that:
\begin{itemize}
\item the induced morphism $\mathcal{B}\rightarrow T$ is flat, 
\item $(\mathcal{X}_0,\mathcal{B}_0)\cong (X,B)$, and
\item for any $t\in T$, $(\mathcal{X}_t,\mathcal{B}_t)$ is a pair.
\end{itemize}
}
\end{definition}

\begin{definition}{\em 
A set of couples $\mathcal{C}$ is called {\em bounded} if there exists a projective morphism $\mathcal{X}\rightarrow S$ of varieties of finite type and a reduced divisor $\mathcal{D}\subset \mathcal{X}$, such that for every element $(X,D)\in \mathcal{C}$, there is a closed point $s\in S$ and an isomorphism $f \colon (X,D) \simeq (\mathcal{X}_s,\mathcal{D}_s)$.

A set of pairs $\mathcal{C}$ is called {\em log bounded} if
$$\{(X,\Supp B)\mid(X,B)\in\mathcal{C}\}$$
is a bounded set of couples.

A set of algebraic varieties $\mathcal{C}$ is called {\em bounded} if 
$$\{(X,0)\mid X\in\mathcal{C}\}$$
is a bounded set of couples. 

A set of algebraic germs $\mathcal{C}$ is called {\em analytically bounded} if there exists a projective morphism $\mathcal{X}\rightarrow S$ of varieties of finite type, such that for every element $(x\in X)\in\mathcal{C}$, there is a closed point $s\in S$ and a closed point $x_s \in \mathcal{X}_s$ such that 
$$\hat{\mathcal{O}}_{X,x}\cong\hat{\mathcal{O}}_{X_s,x_s}.$$

A set of pairs $\mathcal{C}$ is called {\em log bounded up to deformation} if there exists a log bounded family $\mathcal{C}_0$ such that for every element $(X,B)\in \mathcal{C}$ there exists a deformation $(\mathcal{X},\mathcal{B}) \rightarrow \mathbb{A}^1$ and $t\in \mathbb{A}^1$ such that
$(\mathcal{X}_t,\mathcal{B}_t) \simeq (X,B)$
and $(\mathcal{X}_0,\mathcal{B}_0)\in \mathcal{C}_0$.
}
\end{definition}

\begin{definition}\label{def:log-fano}{\em 
A pair $(E,\Delta_E)$ is called {\em log Fano} if $(E,\Delta_E)$ is klt and $-(K_E+\Delta_E)$ is an ample $\qq$-divisor.}  A variety $E$ is called {\em Fano} if $(E,0)$ is log Fano.
\end{definition}

We recall the boundedness of $\epsilon$-lc log Fano varieties, which was known as the BAB conjecture and is proved by Birkar:

\begin{theorem}{\rm(\cite[Theorem 1.1]{Bir16b})}\label{thm:bab}
Let $n$ be a positive integer and $\epsilon$ be a positive real number. Then 
$$\{X \mid \dim X=n, (X,\Delta)\ \textit{is}\ \epsilon\textit{-lc log Fano for some}\ \Delta\}$$
is bounded.
\end{theorem}

\begin{proposition}\label{bounding-log-pairs-BAB}
Let $n$ be a positive integer, $\epsilon$ and $c$ be two positive real numbers. Let $\mathcal{D}$ be the set of $n$-dimensional pairs $(E,\Delta_E)$, such that
\begin{itemize}
    \item $E$ is $\epsilon$-lc log Fano,
    \item $-(K_E+\Delta_E)$ is pseudoeffective, and 
    \item the non-zero coefficients of $\Delta_E$ are at least $c$.
\end{itemize}
Then $\mathcal{D}$ is a log bounded.
\end{proposition}

\begin{proof}
By Theorem \ref{thm:bab}, for any $(E,\Delta_E)\in\mathcal{D}$, $E$ is bounded. Thus there exists a positive real number $C$, such that for any $(E,\Delta_E)\in\mathcal{D}$, there exists a very ample divisor $A_E$ on $E$, such that $-K_E\cdot A^{n-1}\leq C$. By our assumptions, we have
\[
\Supp(\Delta_E)\cdot A^{n-1} 
\leq
\frac{\Delta_E \cdot A^{n-1}}{c}
\leq
\frac{ -K_E \cdot A^{n-1}}{c} 
\leq
\frac{C}{c}.
\]
By~\cite[Lemma 3.7]{Ale94}, $(E,\Delta_E)$ is log bounded.
\end{proof}

\subsection{Boundedness of Cartier indices}

In this subsection we recall and prove several results on boundedness of Cartier indices.

\begin{lemma}{\rm (\cite[Proposition 6.2]{PS01})}\label{lem:extendcomplement}
	Let $n,m$ be two positive integers, $(Y,E)$ a pair of dimension $n$, such that $E$ is a reduced divisor and $(Y,E)$ is plt near $E$. Suppose $Y\rightarrow X$ is a contraction and $x\in X$ a closed point. Assume that
	\begin{enumerate}
		\item $-(K_Y+E)$ is big and nef over $X$, and
        \item there is an $m$-complement $K_E+{\rm Diff}_E(0)^{+}$ of the $\mathbb R$-divisor $K_E+{\rm Diff}_E(0)$ over $z$.
		\end{enumerate}
		
	Then there is an $m$-complement $K_Y+E+B$ of $K_Y+E$ over $z$ such that ${\rm Diff}_E(0)^{+}={\rm Diff}_E(B)$. 
	
	In particular, if $Y=X$ and and $m(K_Y+E)|_{E}$ is Cartier near $x$, then $m(K_Y+E)$ is Cartier near $x$. Furthermore, by Noetherian property, if $m(K_Y+E)|_{E}$ is Cartier, then $m(K_Y+E)$ is Cartier near $E$.
\end{lemma}

\begin{lemma}\label{lem:cartier-index-ky+e}
Let $n$ be a positive integer, $\epsilon,\delta$ two positive real numbers. Then there exists a positive integer $m$ depending only on $n,\epsilon$ and $\delta$ satisfying the following.
Assume $\pi \colon Y\rightarrow X$ is a $\delta$-plt blow-up of an $n$-dimensional $\epsilon$-lc singularity $x\in X$ that extracts a unique exceptional divisor $E$. Then $m(K_Y+E), mE$ are Cartier near $E$ and $mK_X$ is Cartier near $x$.
\end{lemma}

\begin{proof}
Since all the coefficients of ${\rm Diff}_{E}(0)$ are of the form $\frac{l-1}{l}$ for some positive integer $l$, and since $(E,{\rm Diff}_{E}(0))$ is $\delta$-klt by adjunction formula, all the coefficients of ${\rm Diff}_{E}(0)$ are contained in a finite set of rational numbers By Proposition \ref{bounding-log-pairs-BAB}, $(E,{\rm Diff}_{E}(0))$ is log bounded, hence there exists a positive integer $m_0$ such that $m_0(K_E+{\rm Diff}_{E}(0))$ is Cartier. By Lemma \ref{lem:extendcomplement}, $m_0(K_Y+E)$ is Cartier near $E$.

By \cite[Theorem 2]{Mor18b} or \cite[Corollary 1.10]{HLS19}, there exists a positive integer $m_1$ such that $m_1K_X$ is Cartier near $x$. In particular, $m_0m_1(K_Y+E)$ and $m_0m_1(K_Y+(1-a_E(X,0))E)$ are both Cartier near $E$. Thus $m_0m_1a_E(X,0)E$ is Cartier. Since $x\in X$ is $\epsilon$-lc, $a_E(X,0)\geq\epsilon$. By \cite[Proposition 4.2]{HLS19}, $a_E(X,0)$ is contained in a finite set, thus there exists an integer $m$ such that $m_0m_1|m$ and $mE$ is Cartier. In particular, $m(K_Y+E)$ is Cartier near $E$ and $mK_X$ is Cartier near $x$.
\end{proof}

\subsection{Cone singularities}\label{subsec:cone}
In this subsection, we recall the definition of cone singularities, and prove some basic properties regarding isotropies and discrepancies of cone singularities.

\begin{definition}{\em 
A {\em cone singularity} is a normal affine algebraic variety $X$ with an effective $\mathbb{K}$-action such that
$X$ has a unique fixed closed point $x\in X$, called the {\em vertex of the action}, and any orbit closure of this $\mathbb K$-action contains $x$.
A pair $(X,B)$ is called a {\em cone singularity} at $x\in (X,B)$ if $x\in X$ is a cone singularity and $B$ is a $\kk$-invariant $\mathbb R$-divisor.
}
\end{definition}

\begin{definition}{\em 
Let $x\in X$ be a cone singularity. For any closed point $y\not=x$ contained in $X$, it is well known that the {\em stabilizer} of the torus action on the orbit $\kk^*_y$ of $y$ 
is $\mu_d \subset \kk^*_y$, i.e. the subgroup of $d$-th roots of unit, for some positive integer $d$.
In such case, we say that $\kk$ acts with {\em isotropy} at $y$ equal to $d$. 
If $d=1$, then we say that $\kk$ acts with trivial isotropies at the point $y$.
We say that $x\in X$ has {\em isotropies bounded by N} if for any point $y\not=x$ contained in $X$, the isotropy of the $\mathbb K$-action at $y$ is at most $N$.
}
\end{definition}

The following theorem is a standard theorem of the theory of $\mathbb{T}$-varieties (see, e.g.~\cite{AH05,AHS08,AIPSV12}).

\begin{theorem}\label{demazure-isomorphism-with-boundary}
Let $x\in (X,B)$ be a cone singularity.
Then there exists a projective variety $E$ and an ample $\qq$-Cartier $\mathbb Q$-divisor $D_E$ on $E$ such that
\[
X\simeq {\rm Spec}\left( \bigoplus_{k \geq 0} H^0(E, kD_E) \right),
\]
and under this isomorphism $x\in X$ corresponds to the maximal ideal
\[
m_x := \bigoplus_{k>0} H^0(E,kD_E).
\]
Moreover, there exists a a good quotient $\rho \colon X\setminus\{x\} \rightarrow E$ (in the sense of~\cite[Definition 2.3.1]{ADHL15}) for the torus action 
and an effective $\mathbb R$-divisor $B_E$ on $E$ such that $B$ is the closure of $\rho^*(B_E)$.
\end{theorem}

\begin{proof}
The isomorphism is proved in~\cite[3.5]{Dem88} for some variety $E$ and an ample $\qq$-divisor $D_E$ on $E$.
Since the action of $\kk$ on $X$ has a unique fixed point and every orbit closure contains $x\in X$, we conclude that the weighted monoid of the action is pointed, and hence it is isomorphic to $\mathbb{Z}_{\geq 0}$ as the torus action is one-dimensional.
In particular, $E$ is projective over ${\rm Spec}(\kk)$.
By~\cite[Construction 1.6.13]{ADHL15}, there is a good quotient $\rho 
\colon X \setminus \{ x\} \rightarrow E$.
Let $B_E:=\rho_*(B)$ and the proof is finished.
\end{proof}

\begin{notation}{\em 
Given a cone singularity $x\in X$, we denote $\widetilde{X}\rightarrow X$ the blow-up at the vertex, $F_{\widetilde{X}} \subset \widetilde{X}$ the exceptional divisor of $\widetilde{X}\rightarrow X$, and $\rho \colon \widetilde{X}\rightarrow E$ the induced good quotient. In Lemma~\ref{exceptional-vs-quotient}, we will see that $E$ is indeed isomorphic to $F_{\widetilde{X}}$ in the case of klt cone singularities.
}
\end{notation}

\begin{definition}{\em 
Let $x\in (X,B)$ be a $\qq$-Gorenstein klt cone singularity, $E$ is the Chow quotient of $x\in X$. 
Let $D_E$ and $B_E$ be as in Theorem~\ref{demazure-isomorphism-with-boundary}.
We may write $D_E = \sum_{Z} \frac{p_Z}{q_Z} Z$, where the sum runs over all prime divisors
$Z$ on $E$ and ${\rm gcd}(p_Z,q_Z)=1$. We may define the boundary divisor
\[
\Delta_{E}:= \sum_{Z\subset E} \left( 1- \frac{1}{q_Z} \right) Z.
\]
Since each $q_Z$ only depends on the cone singularity $x\in X$, $\Delta_{E}$ only depends on the cone singularity $x\in X$.

The pair $(E,\Delta_E+B_E)$ is called the {\em log Fano quotient} of the cone singularity $x\in (X,B)$,
while the triple $(E,D_E;B_E)$ is called the {\em associated triple} of the cone singularity.
Notice that the associated triple is well-defined modulo linear equivalence of $D_E$.

We also call $D_E$ the {\em $\qq$-polarization} of the cone singularity, and $B$ the {\em cone over the divisor} $B_E$.
}
\end{definition}

We need the following result on Cartier index of the $\mathbb Q$-polarization:

\begin{proposition}{\rm (\cite[Proposition 1.3.5.7]{ADHL15})}\label{isotropy-control}
Let $x\in X$ be a cone singularity, $y\in X\backslash\{x\}$ a closed point, and $\rho \colon X\setminus\{x\} \rightarrow E$ the induced good quotient. 
Then the isotropy at $y$ equals to the Cartier index of $D_E$ at $\rho(y)\in E$.
\end{proposition}

\begin{proof}
Replacing $E$ with a suitable affine neighborhood of $y$ we may assume $X$ is affine.
The character group of the orbit of $x$ is isomorphic to $\zz D_E / \zz mD_E$ where 
$m$ is the smaller positive integer such that $f(y)\neq 0$ for some $f\in \Gamma(E,mD_E)$.
Thus, the isotropy group at $x$ is isomorphic to $\zz_m$.
\end{proof}

\begin{proposition}\label{singularities-base}
Let $x\in (X,B)$ be a $\qq$-Gorenstein klt cone singularity with isotropies bounded by $N$ such that $B$ is a $\mathbb Q$-divisor.
Assume that $(X,B)$ is $\epsilon$-lc at $x$. 
Then $(E,\Delta_E+B_E)$ is $\epsilon/N$-log Fano.
\end{proposition}

\begin{proof}
By \cite[Proposition 3.11]{PS11}, every $\kk$-invariant Cartier divisor on an affine cone singularity is principal.
Moreover, from~\cite[Remark 3.7]{PS11} we know that the field of fractions of $X$ is isomorphic to $\kk(Y)[M]$,
where $\kk(Y)$ is the field of fractions of $Y$ and $M$ is the lattice of torus characters.
Hence, we can write 
\begin{equation}\label{principal}
m (K_X+B) = {\rm div}_X( f \chi^{u}),
\end{equation}
where $m$ is the Cartier index of $K_X+B$, $f$ is a rational function on $Y$, and $u \in M$
is contained on the weighted monoid of the action of $\kk$ at $x\in X$.
Pushing-forward equation~\eqref{principal} via $\rho$,
we obtain the equation
\begin{equation}\label{equation-on-Y}
m(K_E+\Delta_E+B_E) = -u D_E + H
\end{equation}
where $H={\rm div}_Y(f)$.
In particular, 
$-(K_E+\Delta_E+B_E)$ is an ample $\qq$-divisor. 

We claim that $(E,\Delta_E+B_E)$ is $\epsilon/N$-lc.
Let $(E,D_E;B_E)$ be the associated triple of $x\in (X,B)$.
For any birational morphism $f\colon W\rightarrow E$, let $\widetilde{X}_W$ be the relative spectrum of the divisorial sheaf $\bigoplus_{k\geq 0}\mathcal{O}_{W}(kf^*(D_E))$.
We have a commutative diagram 
\begin{equation}\label{commutative-diagram}
 \xymatrix{
  \widetilde{X}_W \ar[r]^-{\widetilde{f}}\ar[d]_-{\rho_W}   & \widetilde{X}\ar[r] \ar[d]_-{\widetilde{\rho}} & X \ar@{-->}[ld]^-{\rho} \\
  W \ar[r]^-{f} & E & 
 }
\end{equation}
where $\rho_W,\widetilde{f}$ and $\widetilde{\rho}$ are induced morphisms. For any prime divisor $F$ on $W$, let $w_F$ be the Weil index of $f^*D_E$ at $F$, i.e. the smallest positive integer such that $w_Ff^*D_E$ is a Weil divisor at $F$.
By~\cite[Proposition 3.14]{PS11} and~\cite[Theorem 2.8]{Wat81}, we have
\begin{equation}\label{comparison-log-discrepancies}
a_F(E, \Delta_E+B_E) = \frac{1}{w_F} a_{\rho_W^{-1}(F)}(X,B).
\end{equation}
By Proposition~\ref{isotropy-control}, the divisor $ND_E$ is Cartier. Since $f^*D_E = f^*(ND_E)/N$, the Weil index $w_F$ of $f^*D_E$ at $F$ is at most $N$.
Therefore $w_F\leq N$, and the proof is finished.
\end{proof}

\begin{definition}{\em 
Let $x\in (X,B)$ be a $\mathbb Q$-Gorenstein klt cone singularity such that $B$ is a $\mathbb Q$-divisor, and let $(E,\Delta_E+B_E)$ be its log Fano quotient. By equation~\eqref{equation-on-Y}
there exists a rational number $r$ so that 
\[
D_E \sim_\qq -r(K_E+\Delta_E+B_E).
\]
We say that $r\in \qq_{>0}$ is the {\em Fano angle} of $x\in (X,B)$ and is denoted by $r(x,X,B)$. We also define the Fano angle of $x\in X$ to be $r(x,X,0)$.
It is clear that $r(x,X,B)\geq r(x,X,0)$ for any klt pair $(X,B)$

}
\end{definition}

The following proposition shows that the log discrepancy of $(X,B)$ at the exceptional divisor obtained
by blowing-up the vertex of the cone singularity is the inverse of the Fano angle of the cone singularity.

\begin{proposition}\label{angle-vs-log-discrepancy}
Let $x\in (X,B)$ be a $\mathbb Q$-Gorenstein klt cone singularity such that $B$ is a $\mathbb Q$-divisor
Then we have that
\[
a_{F_{\widetilde{X}}}(X,B) = \frac{1}{r(x,X,B)}.
\]
Here, $F_{\widetilde{X}}$ is the exceptional divisor extracted by blowing-up the vertex $x\in X$.
\end{proposition}

\begin{proof}
By equation~\eqref{principal}, we can write
\[
m(K_X+B) = {\rm div}_X(f\chi^u)
\]
for some positive integer $m$ and some 
element $u\in M$.
By~\cite[Theorem 2.8]{Wat81} and~\cite[Proposition 3.14]{PS11}, we have that 
\[
a_{F_{\widetilde{X}}}(X,B)= 1+
{\rm coeff}_{F_{\widetilde{X}}}\left( K_{\widetilde{X}}+B_{\widetilde{X}} - \frac{1}{m} {\rm div}_{\widetilde{X}}(f\chi^u) \right) =
-\frac{u}{m} = \frac{1}{r(x,X,B)}.
\]
\end{proof}

\begin{proposition}\label{prop:control-ld-vertex}
Let $n$ be a positive integer, 
$r$ and $\epsilon$ be two positive real numbers.
Then there exists a positive constant $\epsilon_0$, depending only on $\epsilon$ and $r$, satisfying the following.
For any $n$-dimensional $\mathbb Q$-Gorenstein klt cone singularity $x\in (X,B)$ such that
\begin{itemize}
    \item $x\in X$ has trivial isotropies,
    \item $B$ is a $\mathbb Q$-divisor
    \item $r(x,X,B) \leq r$, and
    \item the log Fano quotient $(E,\Delta_E+B_E)$ of $(X,B)$ is $\epsilon$-lc,
\end{itemize}
then $x\in (X,B)$ is $\epsilon_0$-lc.
\end{proposition}

\begin{proof}
Let $f\colon W\rightarrow E$ be a log resolution of the pair $(E,\Delta_E+B_E)$.
Let $(E,D_E;B_E)$ be the associated triple of the cone singularity $x\in (X,B)$.
By~\cite[Example 2.5]{LS13}, the relative spectrum $\widetilde{X}_W$ of the divisorial sheaf
$\bigoplus_{k\geq 0}\mathcal{O}_{E'}(kf^*(D_E))$ is log smooth,
and there is an induced log resolution $\widetilde{X}_W\rightarrow X$ of the pair $(X,B)$.
By~\cite[Proposition 3.14]{PS11} and~\cite[Theorem 2.8]{Wat81}, we have
\[
\epsilon \leq a_F(E,\Delta_E+B_E) = a_{\rho_W^{-1}(F)}(X,B),
\]
for every prime divisor $F$ on $W$.

On the other hand, any divisor on $\widetilde{X}_W$ which is exceptional over $X$, has either the form $\rho^{-1}_W(F)$ for some prime divisor $F$ on $W$, or is the strict transform
of the unique prime exceptional divisor $F_{\widetilde{X}}$ of $\widetilde{X}\rightarrow X$.
Thus, by Proposition~\ref{angle-vs-log-discrepancy},
we conclude that for any prime divisor $F$ over $X$ we have that
\[
a_F(X,B) \geq \min\left\{ 
\epsilon, \frac{1}{r}
\right\}.
\]
Let $\epsilon_0:=\min\{\epsilon,\frac{1}{r}\}$ and the proof is finished.
\end{proof}

\begin{lemma}\label{lem:control-isotropies}
Let $x\in X$ and $x'\in X'$ be two cone singularities
and $\phi \colon X\rightarrow X'$ a $\kk^*$-equivariant cyclic quotient of degree $d$. 
If $x'\in X$ has isotropies at most $N$, then
$x\in X$ has isotropies at most $dN$.
\end{lemma}

\begin{proof}
Let $y\not=x$ be a closed point that is contained in $X$ and let $y':=\phi(y)$. Let $\kk^*_y$ be the orbit of the $\kk$-action corresponding to the point $y$.
The restriction of $\phi \colon X\rightarrow X'$ to $\kk^*_{y}$ induces the quotient $\kk^*_{y} \rightarrow \kk^*_{y}/\mu_{e} \simeq \kk^*_{y'}$ of the torus by the group of $e$-th roots of unit, where $e\leq d$. 
Thus, the isotropy at $y$ is at most $e$ times the isotropy at $y'$, and the proof is finished. \end{proof}

\begin{lemma}\label{exceptional-vs-quotient}
Let $x\in (X,B)$ be a klt cone singularity and $\widetilde{X}\rightarrow X$ the blow-up of the vertex with exceptional divisor $F_{\widetilde{X}}$.
Then the pair obtained by adjunct $K_{\widetilde{X}}+F_{\widetilde{X}}+B_{\widetilde{X}}$ to $F_{\widetilde{X}}$ is isomorphic to the log Fano quotient of $x\in (X,B)$.
\end{lemma}

\begin{proof}
Let $(E,D_E;B_E)$ be the associated triple of $x\in (X,B)$ and $(E,\Delta_E+B_E)$ the log Fano quotient.
Let $x'\in (X',B')$ be the cone singularity corresponding to the triple $(E,ND_E;B_E)$, where $N$ is the Cartier index of $D_E$ which induces a $\kk$-equivariant  morphism $\phi \colon X\rightarrow X'$ of degree $N$.
Let $\widetilde{X}'$ be the relative spectrum of $\bigoplus_{k\geq 0}\mathcal{O}_E(kND_E)$, which is an $\mathbb{A}^1$-bundle over $E$. 
Let $F_{\widetilde{X}'}$ be the exceptional divisor extracted by the morphism $\widetilde{X}'\rightarrow X'$ and $\rho'\colon \widetilde{X}'\rightarrow E$ the good quotient morphism.
We have a commutative diagram:
\begin{equation}\label{finite-quotient-diagram}\nonumber
 \xymatrix{
  & E \\ 
 \widetilde{X} \ar[r]_-{\widetilde{\phi}}\ar[ru]^-{\widetilde{\rho}}\ar[d] & \widetilde{X}'\ar[d]\ar[u]_-{\rho'}\\
 X\ar[r]^-{\phi} & X' \\
 }
\end{equation}

Let
\[
\Delta_{\widetilde{X}'} := {\rho'}^*(\Delta_E), \qquad
B_{\widetilde{X}'} := {\rho'}^*(B_E),
\]
then
\[
\left(\widetilde{X}', F_{\widetilde{X}'} + \Delta_{\widetilde{X}'}+B_{\widetilde{X}'}\right)
\]
is plt. Indeed, it is a locally trivial $\mathbb{A}^1$ bundle over a klt variety, being $F_{\widetilde{X}'}$ the zero section. Hence we have that
\[
K_{\widetilde{X}}+F_{\widetilde{X}}+B_{\widetilde{X}}=
\phi^*( K_{\widetilde{X}'}+ F_{\widetilde{X}'} + \Delta_{\widetilde{X}'}+B_{\widetilde{X}'}),
\]
is also plt. In particular, $F_{\widetilde{X}}$ is normal. Since $\rho \colon F_{\widetilde{X}}\rightarrow E$ is a bijection between normal varieties, it is an isomorphism. We then have
Under this isomorphism we have that 
\[
(K_{\widetilde{X}}+F_{\widetilde{X}})|_{F_{\widetilde{X}}} \simeq K_E+\Delta_E, \text{ and }
B_{\widetilde{X}}|_{F_{\widetilde{X}}} \simeq B_E.
\]
In particular, 
$$(K_{\widetilde{X}}+F_{\widetilde{X}}+B_{\widetilde{X}})|_{F_{\widetilde{X}}}\simeq K_E+\Delta_E+B_E$$
and the proof is finished.
\end{proof}

\begin{remark}{\em 
By Lemma \ref{exceptional-vs-quotient}, from now on we may identify $F_{\widetilde{X}}$ with $E$.}
\end{remark}

\begin{notation}{\em 
We denote $\mathcal{C}_{n}(\epsilon,N)$  the set of $n$-dimensional $\mathbb Q$-Gorenstein $\epsilon$-lc cone singularities $x\in (X,B)$ with isotropies bounded by $N$.
When $\epsilon$ is a positive real number, according to \cite[Theorem 1]{Mor18b},
$\mathcal{C}_{n}(\epsilon,N)$ is bounded.
}
\end{notation}

\begin{corollary}\label{blow-up-vertex}
Let $(x\in (X,B))\in\mathcal{C}_{n}(\epsilon,N)$ be a cone singularity.
Then the blow-up of $x$ is an $\epsilon/N$-plt blow-up. 
\end{corollary}

\begin{proof}
By Proposition~\ref{singularities-base}, the log Fano quotient $(E,\Delta_E+B_E)$ is 
$\epsilon/N$-lc, and by Lemma~\ref{exceptional-vs-quotient},
$$(K_{\widetilde{X}}+F_{\widetilde{X}}+B_{\widetilde{X}})|_{F_{\widetilde{X}}}= K_E+\Delta_E+B_E,$$ 
hence it is $\epsilon/N$-lc. By inversion of adjunction, we conclude that $(\widetilde{X},F_{\widetilde{X}}+B_{\widetilde{X}})$ is $\epsilon/N$-lc.
\end{proof}

The following Lemma is the claim in Step 6 of the proof of~\cite[Theorem 1]{Mor18b}.

\begin{lemma}\label{bounding-family-cone-singularities}
Let $n$ and $N$ be two positive integers, and $\epsilon$ be a positive real number. Let $x_i \in X_i$ be a sequence of cone singularities contained in $\mathcal{C}_{n}(\epsilon,N)$. Up to passing to a subsequence, we can find a morphism $\mathcal{E}\rightarrow S$, a divisor $\mathcal{D}$ on $\mathcal{E}$, and a sequence of closed points $s_i\in S$, such that 
\[
X_i \simeq {\rm Spec}\left(
\bigoplus_{k \geq 0}H^0(\mathcal{E}_{s_i}, k \mathcal{D}_{s_i})
\right)
\]
holds for each $i$.
\end{lemma}

\begin{lemma}\label{lem:log-bounded-cones}
Let $n$ and $N$ be two positive integers, and $\epsilon$ and $c$ be two positive real numbers. 
Then the set of klt pairs $(X,B)$ such that
\begin{itemize}
    \item $X\in \mathcal{C}_{n}(\epsilon,N)$ with vertex $x\in X$
    \item the coefficients of $B$ are at least $c$ near $x$, and
    \item the blow-up of the vertex is a plt blow-up for $(X,B)$,
\end{itemize}
is log bounded near a neighborhood of $x\in X$.
\end{lemma}

\begin{proof}
By~\cite[Theorem 3.10]{Ale94}, it suffices to show that every sequence $x_i\in (X_i,B_i)$ as in the statement, contains a subsequence which is log bounded.
Let $\pi_i \colon \widetilde{X}_i \rightarrow X_i$ be the blow-up of $x_i$. 
By Lemma~\ref{blow-up-vertex}, $\pi_i$ is an $\epsilon/N$-plt blow-up and $-(K_{\widetilde{X}_i}+E_i+B_{\widetilde{X}_i})$ is an ample divisor over $X_i$, where $E_i$ is the unique exceptional divisor of $\pi_i$, and $B_{\widetilde{X}_i})$ is the strict transform of $B_i$ on $\widetilde{X}_i$. By Proposition~\ref{bounding-log-pairs-BAB}, the log Fano quotient $(E_i,\Delta_{E_i}+B_{E_i})$ is log bounded.
By Lemma~\ref{exceptional-vs-quotient}, $E_i$ is isomorphic to the Chow quotient of $X_i$ for the torus action.
By Lemma~\ref{bounding-family-cone-singularities}, we can find a morphism
$\mathcal{E}\rightarrow S$, a divisor
$\mathcal{D}\subset \mathcal{E}$, and a sequence $s_i \in S$, such that
\[
X_i \simeq {\rm Spec} \left( 
\bigoplus_{k\geq 0}H^0(\mathcal{E}_{s_i}, k\mathcal{D}_{s_i})
\right)
\]
for each $i$. Since $\mathcal{E}_i \simeq E_{s_i}$, there is a boundary divisor $\mathcal{B}_{\mathcal{E}} \subset \mathcal{E}$ such that $\Supp( B_{E_i} ) \subset \Supp(\mathcal{B}_{\mathcal{E},s_i})$ for each $i$, with the identification given by the above isomorphism. 
Indeed, the above follows from the log boundedness of the pairs $(E_i,\Delta_{E_i}+B_{E_i})$.
Let 
\[
\mathcal{X}:= {\rm Spec}\left( 
\bigoplus_{k\geq 0}H^0(\mathcal{E}/S,
k\mathcal{D})
\right)
\]
and $\widetilde{\mathcal{X}}$ the relative spectrum of the divisorial sheaf
$\bigoplus_{k\geq 0} \mathcal{O}_{\mathcal{E}}(k\mathcal{D})$ over $\mathcal{E}$.
Observe that we have a good quotient
$\widetilde{\mathcal{X}}\rightarrow \mathcal{E}$ for the torus action and a birational contraction $\widetilde{\mathcal{X}} \rightarrow \mathcal{X}$.
Let $\mathcal{B}$ be the push-forward to $\mathcal{X}$ of the pull-back of $\mathcal{B}_{\mathcal{E}}$ to $\widetilde{\mathcal{X}}$.
Possibly passing to a subsequence, we have that 
$\mathcal{X}_{s_i} \simeq X_i$ and 
$\Supp(B_i)\subset \Supp(\mathcal{B}_{s_i})$ for all $i$.
Therefore, the morphism $\mathcal{X}\rightarrow V$ and the $\mathbb R$-divisor
$\mathcal{B}\subset \mathcal{X}$ is a corresponding log bounded family for the log pairs $(X_i,B_i)$.
\end{proof}

\subsection{Deformation to cone singularities}
In this subsection, we recall the degeneration of an 
$(\epsilon,\delta)$-lc singularity to a 
lc cone singularity (see, e.g.~\cite{LX16,LX17}).
Part of the following proposition is already proved in~\cite[2.4]{LX16}.

\begin{proposition}\label{prop:flat-deformation}
Let $\pi \colon Y \rightarrow X$ be a plt blow-up of the pair $(X,B)$ at $x\in X$ with exceptional divisor $E$ such that $mE$ is Cartier.
Then the pair $(X,B)$ deforms to a pair $(X_0,B_0)$ and the following properties hold
\begin{enumerate}
    \item $(X_0,B_0)$ is a cone singularity,
    \item there exists a $\kk$-equivariant cyclic quotient $\phi: X_0 \rightarrow X_0'$ of degree $m$, such that $X_0'$ is the cone singularity with associated triple $(E,-mE|_E;0)$,
    \item we have $\phi^*(K_{X_0'}+\Delta_0'+B_0')=K_{X_0}+B_0$, where $\Delta_0'$ (resp. $B_0'$) is the cone over $\Delta_E$ (resp. $B_E)$, and
    \item the blow-up of the vertex of $X_0$ is a plt blow-up for $(X_0,B_0)$.
\end{enumerate}
\end{proposition}

\begin{proof}
Let $v={\rm ord}_E$ be the valuation over $x\in X$
corresponding to the exceptional divisor of $\pi$.
Possibly shrinking $X$ to a neighborhood of $x$, we may assume $X$ is affine. Let $X={\rm Spec}(A)$, we consider the extended Rees algebra:
\[
\mathcal{R}:= \bigoplus_{k \in \zz}a_k(v)t^{-k}\subset A[t,t^{-1}],
\]
where $a_k(v):= \{ f\in R \mid v(f)\geq k\}$.
By~\cite[Proposition 2.3]{Tei03}, we know that 
$\mathcal{R}$ is faithfully flat over $\kk[t]$, so there is a deformation $\mathcal{X} \rightarrow \mathbb{A}^1$ which central fiber isomorphic to ${\rm gr}_v(R)$, 
and $\mathcal{X}_t \simeq X$ for all $t\in \mathbb{A}^1$.
Let $\mathcal{B}\subset \mathcal{X}$ be 
the strict transform of $B\times \mathbb{A}^1$ with respect to the birational morphism $\mathcal{X}\dashrightarrow X\times \mathbb{A}^1$.
Since $\mathcal{X}\dashrightarrow X\times \mathbb{A}^1$ is equivariant with respect to the torus action, we conclude that $(\mathcal{X},\mathcal{B}) \rightarrow \mathbb{A}^1$ is a deformation of pairs whose central fiber $(X_0,B_0)$ is a cone singularity, proving (1).

We write 
\[
A = {\rm gr}_v(R) = \sum_{k\geq 0} a_k(v)/a_{k+1}(v) = \bigoplus_{k\geq 0} A_k,
\]
and 
\[
A^{(m)} = \sum_{k\geq 0} a_{mk}(v)/a_{mk+1}(v) = \bigoplus_{k\geq 0}A_{mk}.
\]
For any positive integer $k$ we have an exact sequence
\[
0 \rightarrow \mathcal{O}_Y(-(mk+1)E) \rightarrow
\mathcal{O}_Y(-mkE) \rightarrow \mathcal{O}_E(-mkE)
\rightarrow  0.
\]
By Grauert-Riemenschneider theorem we know that
$h^1(Y, \mathcal{O}_Y(-mk+1(E)))=0$, hecenforth
\[
A_{mk}:=\frac{a_{mk}(v)}{a_{mk+1}(v)} = 
\frac{ \pi_* \mathcal{O}_Y(-mkE)}{ \pi_*\mathcal{O}_Y(-(mk+1)E)} =
\frac{H^0(Y,\mathcal{O}_Y(-mkE))}{H^0(Y,\mathcal{O}_Y(-(mk+1)E))} \simeq 
H^0( E, \mathcal{O}_E(-mkE|_E)).
\]
Since $A_0 \simeq \kk$, we conclude that 
\[
\Spec(A^{(m)}) \simeq 
{\rm Spec}\left(
\oplus_{k \geq 0}
H^0(E, \mathcal{O}_E(-mkE|_E))
\right),
\]
proving (2).

 The equivariant finite morphism 
$\phi\colon X_0 \rightarrow X_0'$ is induced by the Veronese embedding
$A^{(m)} \hookrightarrow A$. Henceforth, $\phi$ has degree $m$, and the equality
\[
\phi^*(K_{X_0'}+\Delta_0'+B_0')=K_{X_0}+B_0,
\]
follows from Hurwitz formula and the definition of log Fano quotient, proving (3).

Finally, observe that we have an induced finite morphism
$\widetilde{\phi}\colon \widetilde{X}_0 \rightarrow \widetilde{X}_0'$ which induces a commutative diagram
\begin{equation}\label{diagram-blow-up-deformation}\nonumber
 \xymatrix{
 \widetilde{X}_0 \ar[r]^-{\widetilde{\phi}}\ar[d] & \widetilde{X}_0'\ar[d]\\
 X_0\ar[r]^-{\phi} & X_0' \\
 }
\end{equation}
where $\widetilde{X}_0 \rightarrow X_0$ and $\widetilde{X}_0' \rightarrow X_0'$
are the blow-ups at the vertices $x_0 \in X_0$ and $x_0' \in X_0'$,
with exceptional divisor $\widetilde{F}_0$ and $\widetilde{F}_0'$.
Since $\widetilde{X}_0'\rightarrow E$ is an $\mathbb{A}^1$-bundle,
we know that 
\[
\left(
\widetilde{X}_0', \widetilde{F}_0'+ \widetilde{\Delta}_0' + \widetilde{B}_0' 
\right)
\]
is plt, where $\widetilde{\Delta}_0'$ (resp. $\widetilde{B}_0'$) is the strict transform of $\Delta_0'$ (resp. $B_0'$) on $\widetilde{X}_0'$.
Thus, $(\widetilde{X}_0, \widetilde{F}_0  +\widetilde{B}_0)$ is plt, where $\widetilde{B}_0$ is the strict transform of $B_0$ on $\widetilde{X}_0$, which concludes the proof of the proposition.
\end{proof}

\section{Examples} 
In this section, we give some examples of torus actions with unbounded isotropies, $(\epsilon,\delta)$-lc singularities such that $\epsilon$ and $\delta$ are not tightly related. Moreover, we construct a sequence of threefold $(\epsilon,\delta)$-lc singularities and show that they are not bounded even in the analytic sense.

\begin{example}{\em 
Given a toric variety $X$ with an action of $(\kk^*)^n$, 
we can take sub-tori $\kk^* \subset (\kk^*)^n$ which act on $X$ 
with arbitrarily large isotropy.
In this case, the corresponding GIT quotient will have singularities whose log discrepancies are not bounded away from zero. For instance, taking a smooth germ $(x_1,\dots, x_n)$ 
and the actions 
\[
t\cdot (x_1,\dots,x_n) = (t^{k_1}x_1,\dots, t^{k_n}x_n),
\]
with $k_1,\dots, k_n$ pairwise coprime numbers, 
the corresponding quotients are all possible weighted projective spaces of dimension $n-1$. Henceforth, a fixed germ may admit $\delta$-plt blow-ups with $\delta$ arbitrarily small.
}
\end{example}

\begin{example}\label{epsilon-zero}{\em 
Let $X_d$ be the cone over a rational curve of degree $d$, i.e.
$X_d$ is the spectrum of 
\[
\bigoplus_{k \geq 0} H^0(\pp^1, \mathcal{O}_{\pp^1}(dH)),
\]
where $H$ is the ample class of a point on $\pp^1$.
Then $X_d$ is lc. Indeed, the blow-up $\pi_d \colon Y_d \rightarrow X_d$ of the maximal ideal 
\[
\texttt{m}_{X_d} := \bigoplus_{k>0} H^0(\pp^1, \mathcal{O}_{\pp^1}(dH)),
\]
extracts a unique exceptional divisor $E_d\simeq \pp^1$. It is clear that $(Y_d,E_d)$ is log smooth, and we can write
\[
\pi_d^*(K_{X_d})= K_{Y_d}+ \left(1 - \frac{2}{d} \right) E_d.
\]
Henceforth, the algebraic variety $X_d$ is $2/d$-lc but not $2/d$-klt. In particular, 
\[
{\rm mld}(X_d)=\frac{2}{d}
\]
converges to $0$. 
}
\end{example}

\begin{example}\label{delta-zero}
{\em 
Consider the surface singularities $A_n$. Although they are canonical singularities, we show that any plt blow-up of an $A_n$ singularity is a $\delta$-plt blow-up with $\delta <2/n$.

Let $\pi_n\colon Y_n \rightarrow A_n$ be a plt blow-up, $E_n$ the unique divisor extracted by $\pi_n$, then
$$K_{E_n}+\Delta_{E_n}:=(K_{Y_n}+E_n)|_{E_n}.$$
Since $E_n$ isomorphic to $\pp^1$ and the dual graph of $A_n$ is a chain of $n$ vertices, we have that 
\[
\Delta_{E_n}=(1-1/a) \{0\} + (1-1/b)\{ \infty\}
\]
where $a,b\in\mathbb N$ and $a+b\geq n+1$,
with equality if and only if $E_n$ is an exceptional divisor of the minimal resolution.
In particular, one of $a,b\geq\frac{n}{2}$, hence
$(E_n,{\Delta}_{E_n})$
is not $\frac{2}{n}$-klt. By inversion of adjunction, $\pi_n$ is not a $\frac{2}{n}$-plt blow-up.
}
\end{example}

Now, we turn to give an example of an unbounded sequence of $(\epsilon,\delta)$-lc singularities.
This example shows that the statement of Theorem~\ref{thm:bounded-surfaces} does not hold in higher dimensions.

\begin{example}\label{unbounded}{\em 
Consider the $A_3$ singularity $\{ x^2+y^2+z^3=0 \} \subset \mathbb{A}^3$, and the non-equisingular deformation of the $A_3$ singularity:
\[
X_0 := \{ x^2+y^2+z^3+z^2w = 0 \} \subset \mathbb{A}^4.
\]
Observe that $X_0$ is a cone singularity with the action given by
\[
t\cdot (x,y,z,w) = (t^3x,t^3y,t^2z,t^2w).
\]
Since the restriction of $X_0$ to $w=0$ and
$w\neq 0$ are both klt varieties, by the inversion of adjunction, $X_0$ is klt.

Thus, $X_0$ is $\epsilon$-lc for some $\epsilon$ positive real number.
We will consider the following sequence of deformations of $X_0$:
\[
\mathcal{X}_n := \{ x^2+y^2+z^3+z^2w+tw^n =0\} \subset \mathbb{A}^5
\rightarrow \mathbb{A}^1_{t},
\]
where $n\geq 4$.
Observe that $X_{n,0}\simeq X_0$ for all $n$. 
The above deformations are equisingular in the sense of~\cite{Wah76}.
Moreover, no two such deformations are equivalent (see, e.g.~\cite[Theorem 9.2]{ Har10}).
By Proposition~\ref{blow-up-vertex}, we know that the blow-up 
$\pi_0 \colon Y_0 \rightarrow X_0$ of the vertex of $X_0$ is a $\delta$-plt blow-up for some $\delta>0$.
We denote the exceptional divisor of $\pi_0$ by $E_0$.
Let $\mathcal{Y}_n\rightarrow \mathbb{A}^1_t$ be the flat deformation induced by blowing-up the ideal $\langle x,y,z,w\rangle$ of $\mathcal{X}_n$.
Then, we have that $K_{\mathcal{Y}_n}+\mathcal{E}_n$ is a $\qq$-Cartier divisor.
Indeed, $\mathcal{E}_n$ is Cartier by construction and 
$\mathcal{Y}_n$ is Gorenstein since it has complete itersection singularities.
Observe that 
\[
(K_{\mathcal{Y}_n}+\mathcal{E}_n)|_{Y_0}=
K_{Y_0}+E_0,
\]
therefore $(K_{\mathcal{Y}_n}+\mathcal{E}_n)|_{Y_t}$ is $\delta$-plt for $t$ general.
Thus, $\mathcal{X}_{n,t}$ is $(\epsilon,\delta)$-lc for $t$ general.
We claim that the sequence $\mathcal{X}_{n,t}$ is not analytically bounded.
Indeed, $\mathcal{X}_{n,t}$ is an isolated hypersurface threefold singularity, hence we can compute its Tjurina number: 
\[
{\rm Tju}(\mathcal{X}_{n,t})=\dim_\kk 
\kk[x,y,z,w]/
\langle I_{\mathcal{X}_{n,t}}, {\rm Jac}_{\mathcal{X}_{n,t}}\rangle= n+2,
\]
which forms a diverging sequence. 
Thus, $\mathcal{X}_{n,t}$ does not belong to an analytically bounded family by the upper semicontinuity of the Tjurina numbers.
}
\end{example}

\section{Bounded deformations for $(\epsilon,\delta)$-lc singularities}\label{sec:proof-bounded-deformations}

In this section, we prove that $(\epsilon,\delta)$-lc singularities of fixed dimension are bounded up to a deformation. We give a more general statement which shows the log boundedness of $(\epsilon,\delta)$-lc pairs of fixed dimension.

\begin{theorem}\label{thm:bounded-deformations-boundary}
Let $n$ be a positive integer, $\epsilon,\delta$ and $c$ three positive real numbers.
Then the set of singularities $x\in (X,B)$ such that
\begin{itemize}
\item $X$ is $n$-dimensional $\qq$-Gorenstein at $x\in X$,
\item $(X,B)$ is $(\epsilon,\delta)$-lc at $x\in X$, and
\item the coefficients of $B$ are at least $c$
\end{itemize}
forms a log bounded family up to deformation.
\end{theorem}

\begin{proof}We use notation in Proposition \ref{prop:flat-deformation}. Let $(X,B)$ be a pair and $x\in X$ as above. First we show that we may assume $B$ is a $\mathbb Q$-divisor: indeed, by continuity of log discrepancies, we may find a $\mathbb Q$-divisor $B'$ on $X$, such that $\Supp B'\supset\Supp B$, all the coefficients of $B$ are at least $\frac{1}{2}c$, such that $(X,B')$ is $(\frac{1}{2}\epsilon,\frac{1}{2}\delta)$-lc at $x$. Possibly replacing $B$ by $B'$, $c,\epsilon,\delta$ by $\frac{1}{2}c,\frac{1}{2}\epsilon$ and $\frac{1}{2}\delta$ respectively, we may assume that $B$ is a $\mathbb Q$-divisor.

Let $\pi \colon Y \rightarrow X$ be a $\delta$-plt blow-up of $(X,B)$ at $x$ and $E$ the unique exceptional divisor of $\pi$. Since $X$ is $\mathbb Q$-Gorenstein at $x$, $\pi$ is also a $\delta$-plt blow-up of $x\in X$. Moreover, $X$ is $\epsilon$-lc at $x$.

By Lemma~\ref{lem:extendcomplement} and Lemma~\ref{lem:cartier-index-ky+e}, there exists a positive integer $m$ depending only on $n,\epsilon$ and $\delta$, such that $mE$ is Cartier.
By $(1)$ of Proposition~\ref{prop:flat-deformation},
we know that there exists a flat morphism$\colon \mathcal{X}\rightarrow \kk$ and a divisor $\mathcal{B}\subset \mathcal{X}$ such that for every $t\in \kk^*$ we have
$(\mathcal{X}_t,\mathcal{B}_t)\simeq (X,B)$ and 
$(X_0,B_0) := (\mathcal{X}_0,\mathcal{B}_0)$ is a cone singularity with vertex $x_0$.

First, we claim that the cone singularities $x_0 \in X_0$ belong to a bounded family.
By $(2)$ of Proposition~\ref{prop:flat-deformation}, we have a $\kk^*$-equivariant finite quotient
$X_0\rightarrow X_0'$ of degree $m$,
such that there is an isomorphism
\[
X_0' \simeq {\rm Spec} \left( 
\bigoplus_{k \geq 0}H^0( E, k(-mE|_E))
\right).
\]
Since $-mE|_E$ is a Cartier divisor the $\kk^*$-action on $X_0'$ has trivial isotropies (see, e.g.~\cite[Proposition 1.3.5.7]{ADHL15}) away from the vertex. 
The log Fano quotient of $(X_0',\Delta_{0}')$ is isomorphic to $(E,\Delta_E)$
which has $\delta$-lc singularities.
On the other hand, by~\cite[Theorem 1]{Mor18a} and~\cite[Lemma 2.20]{Mor18a}, we know that we can write 
\[
K_Y+(1-a_E(X,0))E = \pi^*K_X,
\]
where the rational number $a:=1-a_E(X,0)$ belongs to a finite set which only depend on
$n,\epsilon$ and $\delta$.
Thus, we have the $\qq$-linear equivalence 
\[
-mE|_E \sim_\qq -\frac{m}{1-a}(K_E+\Delta_E),
\]
which implies that the Fano angle of $x_0' \in X_0'$ is bounded
by a constant only depending on $m$ and $a$.
Since the rational numbers $m$ and $a$ only depend on $n,\epsilon$ and $\delta$, we conclude that 
the Fano angle of $x_0'\in X_0'$ is bounded by a constant
which only depends on $n,\epsilon$ and $\delta$.
By Proposition~\ref{prop:control-ld-vertex}, we conclude that $x_0' \in X_0'$ is $\epsilon_0$-lc, for some positive constant $\epsilon_0$ which only depends on $n,\epsilon$ and $\delta$.
Thus, by Lemma~\ref{lem:control-isotropies} and $(3)$ of Proposition~\ref{prop:flat-deformation}, we conclude that the isotropies of $x_0 \in X_0$
have an upper bound $m$ and its log discrepancies have a lower bound $\epsilon_0$.
By~\cite[Theorem 1]{Mor18b}, $x_0\in X_0$ belongs to a bounded family.

Finally, we prove that the pairs $(X_0,B_0)$ are log bounded on a neighborhood of $x_0\in X_0$.
By~\cite[Corollary 3.10]{Sho92} we know that there exists a constant $c_0$, only depending on $c$ and $\delta$, such that the coefficients of $B_0$ are at least $c_0$.
By $(4)$ of Proposition~\ref{prop:flat-deformation}, we know that the blow-up of the vertex
$x_0\in X_0$ is a plt blow-up for the pair $(X_0,B_0)$.
Therefore, we can apply Lemma~\ref{lem:log-bounded-cones}
to conclude that the pairs $(X_0,B_0)$ are log bounded on a neighborhood of the vertex $x_0 \in (X_0,B_0)$
\end{proof}

Now, we turn to prove a more general version of Theorem~\ref{lower-bound}, which consider pairs whose boundary has coefficients on a finite set of rational numbers.

\begin{theorem}\label{lower-bound-boundary}
Let $n$ be a positive integer number, $\epsilon$ and $\delta$ two positive real numbers, and $\mathcal{R}$ a finite set of real numbers.
Let $i$ be a lower semicontinuous (resp. upper semicontinuous) invariant of klt singularities.
Then there exists a constant $i_0$, only depending on $n,\epsilon,\delta,\mathcal{R}$ and $i$
such that for every $n$-dimensional $(\epsilon,\delta)$-lc singularity  
$x\in (X,B)$, 
with the coefficients of $B$  belonging to $\mathcal{R}$,
we have $i(x\in (X,B))\geq i_0$
(resp. $i(x\in (X,B))\leq i_0)$.
\end{theorem}

\begin{proof}
Assume $i$ is a lower-semicontinuous invariant of klt singularities. 

Denote by $\mathcal{C}$ the set of $n$-dimensional $(\epsilon,\delta)$-lc pairs $x\in (X,B)$ such that the coefficients of $B$ are contained in $\mathcal{R}$.
By Theorem~\ref{thm:bounded-deformations-boundary}, there exists a log bounded family $\mathcal{C}_0$ of cone singularities $x_0 \in (X_0,B_0)$, such that for every element $(X,B)\in \mathcal{C}$ there is a deformation of pairs $(\mathcal{X},\mathcal{B})\rightarrow \mathbb{A}^1$, such that
$(\mathcal{X}_t, \mathcal{B}_t)\simeq (X,B)$ for any $t\neq 0$, and $(\mathcal{X}_0,{\Supp}\, \mathcal{B}_0)\simeq (X_0,\,\,{\Supp} B_0)$. In the following, we may assume $(\mathcal{X}_0,\mathcal{B}_0) \simeq (X_0,B_0)$ since the coefficients of $B_0$ belong to a finite set. Moreover, the coefficients of $B_0$ belong to a finite set which only depends on $\delta$ and $\mathcal{R}$.
By lower-semicontinuity we have that:
\[
i(x \in (X,B)) = i(x_t \in (\mathcal{X}_t,\mathcal{B}_t)) \geq
i(x_0 \in (\mathcal{X}_0,\mathcal{B}_0)) = 
i(x_0\in (X_0,B_0)).
\]
Moreover, since $x_0\in (X_0,B_0)$ belongs to a log bounded family, we deduce that there exists $i_0$ so that 
$i(x_0 \in (X_0,B_0))\geq i_0$ for all cone singularities in $\mathcal{C}_0$.

Since the family $\mathcal{C}_0$ only depends on $n,\epsilon,\delta$ and $c$, we conclude that $i_0$ only depends on
$n,\epsilon,\delta,c$ and $i$, concluding the proof.

Replacing $i$ by $-i$, we obtain the statement for upper-semicontinuous invariants. 
\end{proof}

\begin{proof}[Proof of Corollary~\ref{upper-bound-mult}]
The proof follows from Theorem~\ref{lower-bound} and the upper semicontinuity of the multiplicity (see, e.g.~\cite{Ben71}).
\end{proof}

\begin{proof}[Proof of Corollary~\ref{upper-bound-analytic-embedding-dimension}]
The proof follows from Theorem~\ref{lower-bound} and~\cite[Proposition 0.35]{Fis76}).
\end{proof}

\section{Boundedness of $(\epsilon,\delta)$-lc surface singularities}\label{sec:proof-bounded-surfaces}\label{sec:4}

In this section, we work over the field of complex numbers $\mathbb C$. We recall a result regarding deformations of surface singularities:

\begin{definition}{\em 
A deformation $\mathcal{X}\rightarrow T$ of a singularity $x\in X$, over a finite dimensional local $\cc$-algebra, is said to be {\em versal} if any other given deformation of $X_0$ over a finite dimensional local $\cc$-algebra can be obtained by base change from $\mathcal{X}\rightarrow T$.
}
\end{definition}

We need the following result by Schlessinger and Pinkham:

\begin{theorem}{\rm(\cite{Sch73,Pin74})}\label{versal-deformation-surfaces}
Let $x\in X$ be an isloated singularity. Then $x\in X$ admits an algebraic versal deformation. Moreover, if $x\in X$ admits a $\cc$-action, then
$\mathcal{X}\rightarrow T$ is $\cc$-equivariant.
\end{theorem}

\begin{proof}[Proof of Corollary~\ref{thm:bounded-surfaces}]
First we prove $(1)$.
By Theorem~\ref{demazure-isomorphism-with-boundary}, for each $x_0\in X_0$ in $\mathcal{C}_{2}(\epsilon_0,N)$, there is an isomorphism 
\[
X_0 \simeq {\rm Spec}\left(
\bigoplus_{k\geq 0} H^0(\pp^1, \mathcal{O}_{\pp^1}(kD))
\right)
\]
where $D$ is an ample $\qq$-divisor on $\pp^1$.
By Proposition~\ref{isotropy-control}, the Cartier index of $D$ is $\leq N$. Thus, the denominators of $\{ D\}$ are bounded by $N$. By Proposition~\ref{singularities-base}, we know that the Fano quotient $(\pp^1,\Delta)$ of $x_0\in X_0$ has at most three fractional coefficients on $\Delta$. Henceforth, $D$ has at most three fractional coefficients as well.
Thus, up to an isomorphism on $\pp^1$ and linear equivalence of $D$, we may assume that
\[
D= \frac{a_0}{N}\{0\} + \frac{a_1}{N}\{1\} + \frac{a_\infty}{N}\{\infty\},
\]
where $a_0$ and $a_1$ are contained in $\zz\cap[0,N]$
and $a_\infty$ is an integer.
We claim that there are finitely many possible values for $a_\infty$.
Indeed, since the singularity $x_0 \in X_0$ is $\epsilon_0$-lc, by Proposition~\ref{angle-vs-log-discrepancy}, its Fano angle is at most $1/\epsilon_0$. Therefore,
\[
a_\infty\leq \frac{2N}{\epsilon_0} - (a_0+a_1).
\]
On the other hand, $D$ is an ample $\qq$-divisor so $a_{\infty}>-(a_0+a_1)$.
Thus, we conclude that there are finitely many possible values for $a_{\infty}$ as 
\[
a_{\infty}\in \zz \cap \left(-(a_0+a_1), \frac{2N}{\epsilon_0}-(a_0+a_1)\right].
\]
which proves (1).

By Theorem~\ref{thm:bounded-deformations-boundary},
for $\epsilon$ and $\delta$ positive real numbers,
$(\epsilon,\delta)$-lc surface singularities are bounded up to deformation, and the central fibers of such deformations are surface cone singularities which belong to a bounded family. Henceforth, there exists $\epsilon_0$ and $N$, depending only on $\epsilon$ and $\delta$, such that an $(\epsilon,\delta)$-lc singularity degenerates to a cone singularity in $\mathcal{C}_{2}(\epsilon_0,N)$, which proves (3).

Finally, for each singularity $x_0 \in X_0$ belonging to $\mathcal{C}_{2}(\epsilon_0,N)$,
by Theorem~\ref{versal-deformation-surfaces},
there exists a space of versal deformations of $X_0$, which we will denote by $\mathcal{X}(X_0) \rightarrow S(X_0)$.
Thus, the morphism of schemes of finite type
\[
\coprod_{X_0 \in \mathcal{C}_{2}(\epsilon_0,N)} 
\mathcal{X}(X_0) \rightarrow S(X_0),
\]
is an analytic bounding family for complex $(\epsilon,\delta)$-lc surface singularities.
\end{proof}

\end{document}